\documentclass[12pt,reqno]{amsart}

\title[Invariant Hilbert scheme resolution: the toric case]{Invariant Hilbert scheme resolution of Popov's $SL(2)$-varieties I: the toric case}
\author{Ayako Kubota}
\address{
Department of Mathematics, Graduate School of Fundamental Science and Engineering, 
Waseda University, 
3-4-1 Ohkubo, Shinjuku, Tokyo 169-8555, Japan}
\email{ayako.kubota.math@gmail.com}

\usepackage{amsthm}
\usepackage[reqno]{amsmath}
\usepackage{amssymb}
\usepackage{amsxtra}
\usepackage[mathscr]{eucal}
\usepackage{enumerate}
\usepackage[all]{xy}
\usepackage[initials,alphabetic,lite]{amsrefs}
\usepackage{youngtab}
\usepackage{graphics}
\usepackage{tikz}

\usepackage{times}
\usepackage{mathptmx}
\usepackage{newtxtext}
\usepackage{newtxmath}
\usepackage{rsfso}

\theoremstyle{plain}
\newtheorem{theorem}[subsection]{Theorem}
\newtheorem{lemma}[subsection]{Lemma}
\newtheorem{proposition}[subsection]{Proposition}
\newtheorem{corollary}[subsection]{Corollary}

\newtheorem*{theorem*}{Theorem}
\newtheorem*{corollary*}{Corollary}
\newtheorem*{MainTheorem}{Main Theorem}
\theoremstyle{definition}
\newtheorem{definition}[subsection]{Definition}
\newtheorem{example}[subsection]{Example}

\theoremstyle{remark}

\newtheorem{remark}[subsubsection]{Remark}

\DeclareSymbolFont{cmletters}{OML}{cmm}{m}{it}
\DeclareSymbolFont{cmsymbols}{OMS}{cmsy}{m}{n}
\DeclareSymbolFont{cmlargesymbols}{OMX}{cmex}{m}{n}
\DeclareMathSymbol{\myjmath}{\mathord}{cmletters}{"7C}
\DeclareMathSymbol{\myamalg}{\mathbin}{cmsymbols}{"71}
\DeclareMathSymbol{\mycoprod}{\mathop}{cmlargesymbols}{"60}
\let\jmath\myjmath

\DeclareMathOperator{\Hilb}{Hilb}
\DeclareMathOperator{\Sym}{Sym}

\DeclareMathOperator{\Spec}{Spec}

\DeclareMathOperator{\pr}{pr}
\DeclareMathOperator{\id}{id}

\DeclareMathOperator{\Hom}{Hom}

\DeclareMathOperator{\diag}{diag}
\DeclareMathOperator{\Gr}{Gr}

\DeclareMathOperator{\Ker}{Ker}

\DeclareMathOperator{\Irr}{Irr}

\DeclareMathOperator{\Univ}{Univ}

\def\E{E_{l,m}}

\def\git{/\!\!/}
\def\G{G_0 \times G_m}
\def\H{\Hilb^{\G}_h(H_{q-p})}
\def\C{\mathbb C}
\def\Z{\mathbb Z \times \mathbb Z/m \mathbb Z}
\def\P{\mathbb P^1 \times \mathbb P^1}
\def\w{\omega}

\makeatletter
\newcommand{\biggg}[1]{{\hbox{$\left#1\vbox to 20.5pt{}\right.\n@space$}}}
\newcommand{\Biggg}[1]{{\hbox{$\left#1\vbox to 23.5pt{}\right.\n@space$}}}
\newcommand{\bigggg}[1]{{\hbox{$\left#1\vbox to 26.5pt{}\right.\n@space$}}}
\newcommand{\Bigggg}[1]{{\hbox{$\left#1\vbox to 29.5pt{}\right.\n@space$}}}
\newcommand{\biggggg}[1]{{\hbox{$\left#1\vbox to 32.5pt{}\right.\n@space$}}}
\newcommand{\Biggggg}[1]{{\hbox{$\left#1\vbox to 35.5pt{}\right.\n@space$}}}
\newcommand{\bigggggg}[1]{{\hbox{$\left#1\vbox to 38.5pt{}\right.\n@space$}}}
\newcommand{\Bigggggg}[1]{{\hbox{$\left#1\vbox to 41.5pt{}\right.\n@space$}}}
\makeatother

\begin{document}

\baselineskip 17pt
\parskip 7pt

\maketitle
\vspace{-1cm}

\begin{abstract}
We show that every $3$-dimensional affine normal quasihomogeneous  $SL(2)$-variety has an equivariant resolution of singularities given by an invariant Hilbert scheme. 
This article treats the case where such $SL(2)$-variety is  toric. The non-toric case is considered in the forthcoming article \cite{Ku}. 
\end{abstract}

%\tableofcontents

\section*{Introduction}
Let $G$ be a reductive algebraic group, and $X$ an affine $G$-scheme of finite type. 
The invariant Hilbert scheme $\Hilb^G_h(X)$ is a moduli space that 
parametrizes $G$-stable closed subschemes of $X$ 
whose coordinate rings have Hilbert function $h$. 
It was introduced by Alexeev and Brion in 
\cite{AB} as 
a common generalization of the $G$-Hilbert scheme of Ito and Nakamura \cite{IN} and the  
multigraded Hilbert scheme of Haiman and Sturmfels \cite{HS}. 
If we take $h$ to be  the \emph{Hilbert function of the general fibers} of  the quotient morphism $\pi : X \to X \git G$,  
we obtain the following morphism:
\[
\xymatrix{
\gamma: \Hilb^G_h(X) \ar[r] &  X \git G.
}
\]
The morphism $\gamma$ is called the \emph{quotient-scheme map}, or the \emph{Hilbert--Chow morphism},  
that  sends a closed point $[Z] \in \Hilb^G_h(X)$ 
to $Z \git G$. 
Since  $\gamma$ (or its restriction to the 
\emph{main 
component} $\mathcal H^{main}$ of the invariant Hilbert scheme $\Hilb^G_h(X)$) is a projective birational morphism, we can ask whether $\gamma$ gives  
a resolution 
of singularities of the quotient variety $X \git G$. 
Becker \cite{Bec} studies the case in which $G=SL(2)$, and $X$ the zero fiber of the moment map of a certain $SL(2)$-action on $(\C^2)^{\oplus 6}$ 
as one of the first examples of the invariant Hilbert scheme with multiplicities, and she proves that $\gamma$ gives a desingularization. 
In \cite{Ter14}, Terpereau considers some cases in which $G$ is a classical group, and $X$ a rational representation of $G$, 
and he provides examples where the Hilbert--Chow morphism $\gamma$ is a desingularization and where it is not. 
Other examples where $\gamma$ gives a resolution can  be found in \cite{Ter2, JR}. 
The present article studies singularities of $3$-dimensional affine normal quasihomogeneous $SL(2)$-varieties via an  invariant Hilbert scheme, and this gives another example where $\gamma$ is a resolution. 

We say that a $G$-variety is \emph{quasihomogeneous} 
if it contains a dense open orbit. 
In \cite{P}, Popov gives a complete classification of $3$-dimensional affine normal quasihomogeneous 
$SL(2)$-varieties:  
they are uniquely determined by a pair of numbers 
$(l,m) \in \{\mathbb Q \cap (0, 1]\} \times \mathbb N$. 
Many of their properties are studied in \cite{P}. For instance, he proves that the corresponding variety $E_{l,m}$ to a pair $(l,m)$ is smooth if and only if $l=1$; 
otherwise $E_{l,m}$ contains a unique $SL(2)$-invariant singular point. 
The varieties $E_{l,m}$ with $l <1$ are one of the simplest non-trivial examples of affine quasihomogeneous varieties with singularities, and what makes the study of such $SL(2)$-varieties interesting and intriguing is the classification 
% beautiful algebro-geometric correspondence 
due to Popov. 
It is known that the singularity of $\E$ 
%an affine normal quasihomogeneous $SL(2)$-variety]
is Cohen-Macaulay but not Gorenstein (\cite{P}, \cite{Pan}, see also \cite{BH}), and its minimal equivariant resolution is given in \cite{Pa}. 
In \cite{BH}, Batyrev and Haddad show by using Popov's classification that every $3$-dimensional affine normal quasihomogeneous $SL(2)$-variety $\E$ can be described as a  categorical quotient of an affine hypersurface $H_{q-p}$ in $\C^5$ modulo an action of $\C^* \times \mu_m$, where we write $l=p/q$ as an irreducible fraction. 
Also, according to \cite{BH}, an $SL(2)$-variety $\E$ admits an action of $\C^*$ and becomes a spherical  $SL(2) \times \C^*$-variety with respect to the Borel subgroup $B \times \C^*$. 
Furthermore, it is shown in \cite{BH} that there is an $SL(2) \times \C^*$-equivariant flip diagram 
\[
\xymatrix@R=5pt{
\E^- \ar@{.>}[rr] \ar[rd] & & \E^+ \ar[ld] \\
& \E &
},
\]
where $\E^{-}$ and $\E^{+}$ are different GIT quotients of $H_{q-p}$ corresponding to some non-trivial characters, and that the varieties $\E$, $\E^-$, and $\E^+$ are dominated by the weighted blow-up $\E'=Bl_O^{\omega}(\E)$ of $\E$ with a weight $\omega$ defined by the $\C^*$-action on $\E$. The weight $\omega$ is trivial if and only if $\E$ is toric.  

This article may be considered as a continuation of above-mentioned works on $\E$, especially of \cite{BH}: we study $3$-dimensional affine normal  
quasihomogeneous $SL(2)$-varieties via an invariant Hilbert scheme by using the GIT quotient description due to Batyrev and Haddad. 
Namely, we study the $SL(2)$-variety $\E$ by means of  the invariant Hilbert scheme $\mathcal H=\Hilb^{\C^* \times \mu_m}_h(H_{q-p})$ associated with the triple $(\C^* \times \mu_m, H_{q-p}, h)$, where $h$ is the Hilbert function of the general fibers of the quotient morphism $\pi : H_{q-p} \to H_{q-p} \git (\C^* \times \mu_m) \cong \E$, and of the corresponding Hilbert--Chow morphism
\[
\xymatrix{\gamma : \mathcal H \ar[r] &  \E},
\]
which is an isomorphism over the dense open $SL(2)$-orbit $\mathfrak U$ in $\E$ (see \S \ref{s-flat}).  
Our main result proves that $\gamma$ gives a desingularization of $\E$. 
The smoothness of  the invariant Hilbert scheme $\mathcal H$ 
is independent of the pair of numbers $(l,m)$, but the behavior of the resolution $\gamma$ does depend on it: it depends on whether $\E$ is toric or not.  
Here we  remark that a necessary and sufficient condition for $E_{l,m}$ being 
a toric variety is given in \cite{G} (see also {\cite[Corollary 2.7]{BH}}) in terms of the numbers $l=p/q$ and $m$: an affine variety $\E$ is toric if and only if $q-p$ divides $m$. 
%\begin{theorem*}[\cite{G}, see also {\cite[Corollary 2.7]{BH}}]
%An affine variety $\E$ is toric if and only if $q-p$ divides $m$. 
%\end{theorem*} 
In both toric and non-toric case, we see that the restriction of $\gamma$ to the main component $\mathcal H^{main}=\overline{\gamma^{-1}(\mathfrak U)}$ of the invariant Hilbert scheme $\mathcal H$ factors equivariantly through the weighted blow-up $\E'$: 
%, and that quasihomogeneous $SL(2)$-varieties $\E$, $\E^-$, and $\E^+$ are dominated by $\mathcal H^{main}$:
%\[
%\xymatrix@R=5pt{
%& & & & \mathcal H^{main} \ar@/_43pt/[lllldddd] \ar@/^18pt/[dddd] \ar[llldd]^{\psi}
%\ar[llldddddd]|(.49)\hole   
%& \\ 
%& & & & & \\
%&  \E' \ar[ldd] \ar[dddd] \ar[rrrdd] & & & & \\
%&&&&& \\
%\E^- \ar[rdd] &  &  & & \E^+ \ar[llldd]  &\\
%&&&&&\\
%& \E & & & & 
%}
%\]
\[
\xymatrix{
\mathcal H^{main} \ar[rr]^{\; \; \; \; \psi} \ar[rrd]_{\gamma|_{\mathcal H^{main}}} &  & \E' \ar[d] \\
&  & \E
}
\]
If $\E$ is toric then we see that $\psi$ is an isomorphism, while if $\E$ is non-toric then we see by an easy observation that $\psi$ is not an isomorphism. 
It is known that if $\E$ is non-toric then the weighted blow-up $\E'$ contains a family of cyclic quotient singularities $\C^2/\mu_b$ (\cite{BH}), and  
therefore a natural candidate for $\mathcal H^{main}$ is the minimal resolution of these quotient singularities. 
\begin{MainTheorem}
%[Corollary \ref{toric sm}, Theorem \ref{maintheorem},  {\cite[Theorem 8.1, Corollary 10.2]{Ku}}]
With the above notation, we have the following. 
\begin{itemize}
\item [\rm(i)] The invariant Hilbert scheme $\mathcal H$ coincides with the main component $\mathcal H^{main}$, and the Hilbert--Chow morphism $\gamma$ is  an equivariant resolution of singularities of $\E$. 
\item [\rm(ii)] If $\E$ is toric, then $\mathcal H$ is isomorphic to the blow-up $Bl_O(\E)$ of $\E$ at the origin. 
\item [\rm(iii)] If $\E$ is non-toric, then $\mathcal H$ is isomorphic to the minimal resolution of the weighted blow-up $\E'$.  
%$\E'=Bl_O^{\omega}(\E)$ of $\E$ with a (non-trivial) weight $\omega$. 
\end{itemize}
\end{MainTheorem}
We need to discuss the toric case and the non-toric case separately, which is because of the difference appeared in items (ii) and (iii) of Main Theorem. 
The present article treats the toric case, and the non-toric case is considered in the forthcoming article \cite{Ku}. 

This article is organized as follows:  
in the first section, we introduce the 
invariant Hilbert scheme and summarize some 
general properties that we use later. 
In \S \ref{s-pop}, we review Popov's classification of $3$-dimensional affine normal quasihomogeneous $SL(2)$-varieties and some related works of Kraft, Panyushev,  Ga{\u\i}fullin, and 
Batyrev--Haddad (\cite{K, Pa, Pan, G, BH}). 
Afterwards, we calculate the Hilbert function of the general fibers of the quotient morphism $\pi : H_{q-p} \to \E$ (Corollary \ref{hilbfunc}). 
\S \ref{s-ideal} is the heart of this article, where 
we show that $\mathcal H^{main}$ contains families of ideals $I_s$ and $J_s$ parametrized by $s \in \C$ (Theorems \ref{idealU} and  \ref{idealD}). 
In \S \ref{s-borel}, we see that $J_0$ is the unique Borel-fixed  point in $\mathcal H$ (Proposition \ref{candidate}), which immediately implies that $\mathcal H$ coincides with $\mathcal H^{main}$ and that $\mathcal H$ is smooth (Corollary \ref{toric sm}).  In the last section, we show that ideals contained in $\mathcal H$ are exactly the ideals $I_s$ and  $J_s$, and their  $SL(2)$-translates (Lemma \ref{bij}). By using Lemma \ref{bij}, we prove that $\mathcal H$ is isomorphic to the blow-up $Bl_O(\E)$ of $\E$ (Theorem \ref{maintheorem}). 
%%%%%%%%%%%%%%%%%%%%%%%%%%%%%
% Section 1
%%%%%%%%%%%%%%%%%%%%%%%%%%%%%
\section{Generalities on the invariant Hilbert scheme}\label{s-hilb}
Brion's survey \cite{B} offers a detailed introduction to the invariant Hilbert scheme. Here we present some definitions and properties on invariant Hilbert schemes that we will use 
later. 

Let $G$ be a reductive algebraic group, and $X$ an affine $G$-scheme 
of finite type. 
We denote by $\Irr(G)$ the set of 
isomorphism classes of irreducible representations of $G$. 
For any $G$-module $V$, we have its isotypical decomposition:
\[
V \cong \bigoplus _{M \in \Irr(G)}\Hom^G(M, V) \otimes M.
\]
We call the dimension of $\Hom^G(M, V)$ the \emph{multiplicity} of 
$M$ in $V$. 
If the multiplicity is finite for every $M \in \Irr(G)$, we can define a function
\[
h_V : \Irr(G) \to \mathbb Z_{\geq 0},\quad  M \mapsto h_V(M):=\dim \Hom^G(M, V), 
\]
which is called the \emph{Hilbert function} of $V$. 

Let $S$ be a Noetherian scheme on which $G$ acts trivially, and 
$Z$ a closed $G$-subscheme of $X \times S$. 
We denote the projection $Z \to S$ by $f$. Then, according to \cite{B}, there is a 
decomposition of $f_*\mathcal O_Z$ as an $\mathcal O_S$-$G$-module
\[
f_*\mathcal O_Z \cong \bigoplus _{M \in \Irr(G)} \mathcal F_M \otimes M,
\]
where sheaves of covariants $\mathcal F_M:=
\mathcal Hom^G_{\mathcal O_S}(M \otimes \mathcal O_S, f_*\mathcal O_Z)$
 are sheaves of $\mathcal O_S$-modules.  
 Assume that each $\mathcal F_M$ is a coherent $\mathcal O_S$-module. 
 Then, each of them is locally-free if and only if it is flat over $S$. 
\begin{definition}[{\cite[Definition 1.5]{AB}}]
Let $h : \Irr(G) \to \mathbb Z_{\geq 0}$ be a Hilbert function. For a given triple $(G, X, h)$, the associated functor
\[
\xymatrix{
\mathcal{H}ilb^G_h(X) : (\mbox{Sch})^{\mbox{op}} \ar[r] & (\mbox{Sets})
}
\]
\[
S \mapsto 
\left\{
\raise20pt\hbox{\xymatrix{
Z \ar[rd]_{f} \ar@{}[r]|*{\subset\; \; } & X \times S \ar[d]^{\mbox{pr}_2}\\
 & S
 }}
 \Bigggggg |\; 
\begin{matrix}
 Z\; \mbox{is a closed}\; G\mbox{-subscheme of}\; X \times S; \quad \\
 f\; \mbox{is a flat morphism};\qquad \qquad \qquad \quad\; \; \\
  f_*\mathcal O_Z \cong \bigoplus_{M \in \Irr(G)} \mathcal F_M \otimes M;\qquad \qquad \; \; \; \;\\
 \mathcal F_M\; \mbox{is locally-free of rank}\; h(M)\; \mbox{over}\; \mathcal O_S
\end{matrix}
 \right\}
\]
is called the \emph{invariant Hilbert functor}.
\end{definition}
\begin{theorem}[{\cite[Theorem 2.11]{B}}]
The invariant Hilbert functor is represented by a quasiprojective scheme 
$\Hilb^G_h(X)$, 
the invariant Hilbert scheme associated with 
the affine $G$-scheme $X$ and the Hilbert function $h$.
We denote by $\Univ^G_h(X) \subset X \times \Hilb^G_h(X)$ the universal family over $\Hilb^G_h(X)$. 
\end{theorem} 
We denote by $T_{[Z]} \Hilb^G_{h}(X)$ the Zariski tangent space to the invariant Hilbert scheme $\Hilb^G_h(X)$ at a closed point $[Z]$. 
We sometimes represent a closed point of $\Hilb^G_h(X)$ by the corresponding ideal $I_Z$ if there is no danger of confusion. 
%as $[I_Z]$ instead of writing $[Z]$, where $I_Z$ stands for the ideal of $Z$. 
\begin{theorem}[{\cite[Proposition 3.5]{B}}]\label{tangent}
With the above notation, we have 
\[
T_{[Z]} \Hilb^G_h(X) \cong \Hom^G_{\C[X]}(I_Z, \C[X]/I_Z). 
\]
\end{theorem}
The invariant Hilbert scheme comes with  a projective morphism called the \emph{quotient-scheme map}, or the \emph{Hilbert--Chow morphism}. 
This is a generalization of the Hilbert--Chow morphism from the $G$-Hilbert scheme $G$-$\Hilb(X)$ to the quotient variety $X/G$ that sends a $G$-cluster to its support.  
The construction of the quotient-scheme map in a general 
setting is explained in \cite[\S 3.4]{B}. 
Here we restrict ourselves to the situation we consider in this article. 
Let 
\[
\xymatrix{
\pi : X \ar[r] &  X \git G:=\Spec(\C[X]^G)
}
\]
be the quotient morphism. 
If $X$ is irreducible, then by the generic flatness theorem,
 $\pi$ is flat over a non-empty 
open subset $Y_0$ of $X \git G$. According to \cite[\S 3.4]{B}, the scheme-theoretic 
fiber of $\pi$ at any closed point of $Y_0$ has the same Hilbert function. 
This special function is called the \emph{Hilbert function of the general fibers of $\pi$}, and we denote it by $h_X$. 
Since $h_X(0)=1$, where $0$ stands for the trivial representation of $G$, the associated 
quotient-scheme map is a morphism
\[
\gamma : \Hilb^G_{h_X}(X) \to X \git G,\quad  [Z] \mapsto Z \git G.
\]
\begin{theorem}[{\cite[Proposition 3.15]{B}}, see also {\cite[Theorem I.1.1]{Bud}}]
\label{commutative}
With the preceding notation, the diagram
\[
\xymatrix{
 \Univ^G_{h_X}(X) \ar[r]^{\quad \; \; \; \; \pr_1} \ar[d]_{\pr_2} & X \ar[d]^{\pi} \\
 \Hilb^G_{h_X}(X) \ar[r]_{\; \; \; \; \; \gamma} & X \git G
}
\]
commutes. Furthermore, the pullback of $\gamma$ to the flat locus $Y_0$ of $\pi$ 
is an isomorphism. 
\end{theorem}
The Zariski closure $\mathcal H^{main}:=\overline{\gamma^{-1}(Y_0)}$ is an irreducible 
component of the invariant Hilbert scheme $\Hilb^G_{h_X}(X)$, and is called the \emph{main component} of $\Hilb^G_{h_X}(X)$ (\cite[Definition 2.4]{Bec}, \cite[Definition 2.3]{Ter}). 
Since the restriction of the quotient-scheme map to the main 
component
\[
\xymatrix{
\gamma|_{\mathcal H^{main}} : \mathcal H^{main} \ar[r] &  X \git G
}
\]
is projective and birational, it is natural to ask whether $\gamma|_{\mathcal H^{main}}$ gives a desingularization of $X \git G$. 

Let us consider a situation that there is an action on $X$ by another connected reductive 
algebraic group $G'$. 
If the action of $G'$ on $X$ commutes with that of $G$, then $G'$ acts both on 
$\Hilb^G_{h_X}(X)$ and on $\Univ^G_{h_X}(X)$,  and every morphism appeared in Theorem \ref{commutative} is $G'$-equivariant (\cite[Proposition 3.10]{B}). 
We especially consider the action of a Borel subgroup $B' \subset G'$ on $\Hilb^G_{h_X}(X)$. 
Let $\mathcal H^{B'} \subset \Hilb^G_{h_X}(X)$ be the set of fixed points for the action of the 
Borel subgroup $B'$. 
\begin{theorem}[{\cite[Lemmas 1.6 and 1.7]{Ter14}}]\label{Borel}
Suppose that $X \git G$ has a unique closed orbit and that this orbit is a point. Then the 
following properties are true. 
\begin{itemize}
\item [{\rm(i)}] Each $G'$-stable closed subset of $\Hilb^G_{h_X}(X)$ contains at 
least one fixed point for the action of the Borel subgroup $B'$. Moreover, 
if $\Hilb^G_{h_X}(X)$ has a unique $B'$-fixed point, then $\Hilb^G_{h_X}(X)$ is connected. 
\item [{\rm(ii)}] The following are equivalent:
\begin{itemize}
\item [\rm{(a)}] $\Hilb^G_{h_X}(X)=\mathcal H^{main}$ and 
$\Hilb^G_{h_X}(X)$ is smooth;
\item [\rm{(b)}] $\dim T_{[Z]}\Hilb^G_{h_X}(X)=\dim \mathcal H^{main}$ for any $[Z] \in \mathcal H^{B'}$, and 
$\Hilb^G_{h_X}(X)$ is connected. 
\end{itemize}
\end{itemize}
\end{theorem}
There is one more useful tool to study the invariant Hilbert scheme. 
According to \cite[\S 4.2]{Bec}, any invariant Hilbert scheme can be embedded into a product of Grassmannians (see also \cite{Ter14}). 
%This is thanks to the construction of the invariant Hilbert scheme 
%as a closed subscheme of the multigraded Hilbert scheme.  
Let $G, X, h_X$, and $G'$ as above. 
For any irreducible representation 
$M \in \Irr(G)$, there is a finite-dimensional $G'$-module $F_M$ that 
generates $\Hom^G(M, \C[X])$ as $\C[X]^G$-modules  (\cite[Proposition 4.2]{Bec}). 
%For such a $G'$-module $F_M$, we can construct a $G'$-equivariant morphism 
%$\eta_M : \Hilb^G_{h_X}(X) \to \Gr(h_X(M), F_M^{\vee})$ 
%in the following way (see \cite[\S 4.2]{Bec} for more details). 
Let $[Z] \in \Hilb^G_{h_X}(X)$, and 
\[
\xymatrix{
f_{M,Z} : F_M \ar[r] &  \Hom^G(M, \C[Z])
}
\] 
be the composition of the inclusion 
$F_M \hookrightarrow  \Hom^G(M, \C[X])$ and the natural surjection $\Hom^G(M, \C[X]) \to \Hom^G(M, \C[Z])$. 
Then, the quotient vector space $F_M/\Ker f_{M,Z}$ defines a point in the Grassmannian $\Gr(h_X(M), F_M^{\vee})$. 
In this way, we obtain a $G'$-equivariant morphism
\[
\eta_M : 
\Hilb^G_{h_X}(X) \to \Gr(h_X(M), F_M^{\vee}), \quad 
[Z] \mapsto  F_M/\Ker f_{M,Z}.
\]
Moreover, there is a finite subset $\mathcal M \subset \Irr(G)$ 
such that the morphism
\[
\xymatrix{
\gamma \times\prod_{M \in \mathcal M} \eta_M : 
\Hilb^G_{h_X}(X) \ar[r] & X \git G \times \prod_{M \in \mathcal M} \Gr(h_X(M), F_M^{\vee})
}
\]
is a closed immersion. 
%%%%%%%%%%%%%%%%%%%%%5
% Section 2
%%%%%%%%%%%%%%%%%%%%%%%%%%%%%
\section{Affine normal quasihomogeneous $SL(2)$-varieties}\label{s-pop}
In \cite{P}, Popov gives a complete classification of affine normal quasihomogeneous $SL(2)$-varieties. 
Consult also the book of Kraft \cite{K}. 
\begin{theorem}[\cite{P}]\label{popov}
Every 3-dimensional affine normal quasihomogeneous $SL(2)$-variety containing more than one orbit is uniquely determined by a pair of numbers 
$(l,m) \in \{\mathbb Q \cap (0,1]\} \times \mathbb N$. 
\end{theorem}
We denote the corresponding variety by $E_{l,m}$. The numbers $l$ 
and $m$ are called the \emph{height} and the \emph{degree} of $E_{l,m}$,
 respectively. 
Write $l=p/q$, where $g.c.d.(q, p)=1$. 
\begin{theorem}[\cite{G}, see also {\cite[Corollary 2.7]{BH}}]\label{toric}
An affine normal quasihomogeneous $SL(2)$-variety $E_{l,m}$ is toric 
if and only if $q-p$ divides $m$. 
\end{theorem}
We use the following notation for some closed subgroups of $SL(2)$:
\begin{align*}
&T:=\left\{
\begin{pmatrix}
t & 0 \\
0 & t^{-1}
\end{pmatrix}
: \;
t \in \C^*
\right\};\quad  
B:=\left\{
\begin{pmatrix}
t & u \\
0 & t^{-1}
\end{pmatrix}
: \; 
t \in \C^*,\; u \in \C
\right\}; \\
%&U:=\left\{
%\begin{pmatrix}
%1 & 0 \\
%u & 1
%\end{pmatrix}
%\Biggg |\; 
%u \in \C
%\right\},\; 
&U_n:=\left\{
\begin{pmatrix}
\zeta & u \\
0 & \zeta^{-1}
\end{pmatrix}
: \; 
\zeta^n=1,\; u \in \C
\right\};\quad 
C_n:=\left\{
\begin{pmatrix}
\zeta & 0 \\
0 & \zeta^{-1}
\end{pmatrix}
: \; 
\zeta^n=1
\right\}.
\end{align*}
An $SL(2)$-variety $E_{l,m}$ is smooth if and only if $l=1$. 
If $l <1$, then $E_{l,m}$ contains three $SL(2)$-orbits: 
the open orbit $\mathfrak U$, a 2-dimensional orbit $\mathfrak D$, and the closed orbit 
$\{O\}$. The fixed point $O$ is a unique $SL(2)$-invariant singular 
point.   
Let
\[
k:=g.c.d.(m, q-p),\quad a:=\frac{m}{k}, \quad  b:=\frac{q-p}{k}.
\]
Then we have 
\[
\mathfrak U \cong SL(2)/C_m,\quad \mathfrak D \cong SL(2)/U_{a(q+p)}.
\]
An explicit construction of the variety $E_{l,m}$ reduces to determine a system 
of generators of the following semigroup (\cite{K}, \cite{Pa}):
\[
M^+_{l,m}:=\left\{ (i,j) \in \mathbb Z^2_{\geq 0}\;  :\;  j \leq li,\; m|(i-j)\right\}.
\]
Let $(i_1, j_1),\; \dots,\; (i_u, j_u)$ be a system of generators of 
$M^+_{l,m}$, and consider a vector
\[
v=(X^{i_1}Y^{j_1},\; \dots,\; X^{i_u}Y^{j_u}) \in V(i_1+j_1) \oplus \dots 
\oplus V(i_u+j_u)=V,
\]
where $V(n):=\Sym^n\langle X, Y\rangle $ is the irreducible $SL(2)$-representation of highest weight $n$. Then, $E_{l,m}$ is isomorphic to the closure $\overline{SL(2) \cdot v} \subset V$.  
\begin{figure}[h]
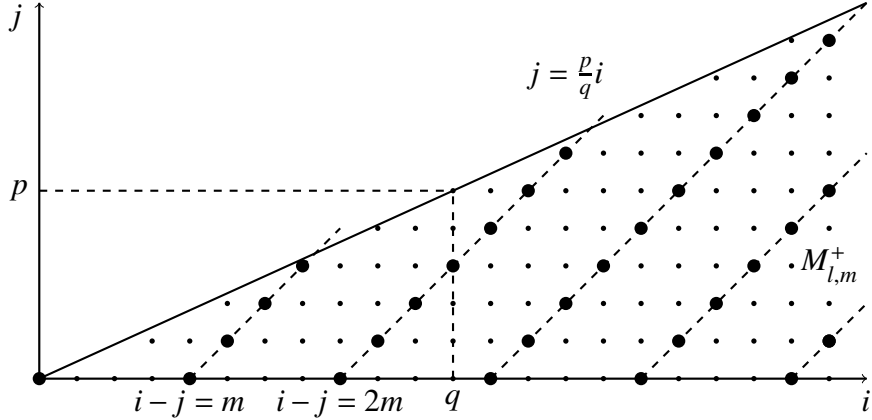

\begin{center}
\tikz{
%\draw [help lines][step = 5mm] 
(-1,-1) grid (11,5);
\draw[thick](0,0)--(0,5);
\draw[thick](0,0)--(11,0);
\draw[thick](0,0)--(11,5);
\draw node at(7,4){$j=\frac{p}{q}i$};
\draw node at(10.5,1.5){$M^+_{l,m}$};
\draw node at(11,-0.3){$i$};
\draw node at(-0.3,4.8){$j$};
\coordinate (J) at (0,5);
\coordinate (I) at (11,0);
\coordinate (O) at (0,0);
\draw[->] (O) -- (J);
\draw[->] (O) -- (I);
%\filldraw[nearly opaque, fill=blue!20](0,0)--(10,4)--(11,4)--(11,0);
%{\fill (0,0) circle (2pt)};
{\draw[dashed][thick](2,0)--(4,2)};
{\draw[dashed][thick](4,0)--(7.5,3.5)};
{\draw[dashed][thick](6,0)--(11,5)};
{\draw[dashed][thick](8,0)--(11,3)};
{\draw[dashed][thick](10,0)--(11,1)};
%{\draw[dashed][thick](12,0)--(12.5,0.5)};
{\draw[dashed][thick](5.5,0)--(5.5,2.5)};
{\draw[dashed][thick](0,2.5)--(5.5,2.5)};
\coordinate (P) at (5.5,0);
 \path (P) node [below] {$q$};
\coordinate (Q) at (0,2.5);
 \path (Q) node [left] {$p$};
\coordinate (A) at (2,0);
       {\fill (A) circle (2.5pt)};
       \path (A) node [below] {$i-j=m$};
\coordinate (B) at (4,0);
       {\fill (B) circle (2.5pt)};
       \path (B) node [below] {$i-j=2m$};
\coordinate (C) at (6,0);
       {\fill (C) circle (2.5pt)};
       %\path (C) node [below] {$i-j=3m$};
\coordinate (D) at (8,0);
       {\fill (D) circle (2.5pt)};
       %\path (D) node [below] {$i-j=4m$};
\coordinate (E) at (10,0);
       {\fill (E) circle (2.5pt)};
       %\path (E) node [below] {$i-j=5m$};
{\fill (0,0) circle (2.5pt)};
{\fill (0.5,0) circle (1pt)};
{\fill (1,0) circle (1pt)};       
{\fill (1.5,0) circle (1pt)};
{\fill (1.5,0.5) circle (1pt)};
{\fill (2,0.5) circle (1pt)};
{\fill (2.5,0) circle (1pt)};
{\fill (2.5,1) circle (1pt)};
{\fill (3,0) circle (1pt)};
{\fill (3,0.5) circle (1pt)};
{\fill (3.5,0) circle (1pt)};
{\fill (3.5,0.5) circle (1pt)};
{\fill (3.5,1) circle (1pt)};
{\fill (3.5,1.5) circle (2.5pt)};
{\fill (4,0.5) circle (1pt)};
{\fill (4,1) circle (1pt)};
{\fill (4,1.5) circle (1pt)};
{\fill (4.5,0) circle (1pt)};
{\fill (4.5,1) circle (1pt)};
{\fill (4.5,1.5) circle (1pt)};
{\fill (4.5,2) circle (1pt)};
{\fill (5,0) circle (1pt)};
{\fill (5,0.5) circle (1pt)};
{\fill (5,1) circle (1pt)};
{\fill (5,1.5) circle (1pt)};
{\fill (5,2) circle (1pt)};
{\fill (5.5,0) circle (1pt)};
{\fill (5.5,0.5) circle (1pt)};
{\fill (5.5,1) circle (1pt)};
{\fill (5.5,1.5) circle (1pt)};
{\fill (5.5,2) circle (1pt)};
{\fill (5.5,2.5) circle (1pt)};
{\fill (6,0.5) circle (1pt)};
{\fill (6,1) circle (1pt)};
{\fill (6,1.5) circle (1pt)};
{\fill (6,2) circle (1pt)};
{\fill (6,2.5) circle (1pt)};
{\fill (6.5,0) circle (1pt)};
{\fill (6.5,0.5) circle (1pt)};
{\fill (6.5,1) circle (1pt)};
{\fill (6.5,1.5) circle (1pt)};
{\fill (6.5,2) circle (1pt)};
{\fill (6.5,2.5) circle (1pt)};
{\fill (7,0) circle (1pt)};
{\fill (7,0.5) circle (1pt)};
{\fill (7,1) circle (1pt)};
{\fill (7,1.5) circle (1pt)};
{\fill (7,2) circle (1pt)};
{\fill (7,2.5) circle (1pt)};
{\fill (7,3) circle (2.5pt)};
{\fill (7.5,0) circle (1pt)};
{\fill (7.5,0.5) circle (1pt)};
{\fill (7.5,1) circle (1pt)};
{\fill (7.5,1.5) circle (1pt)};
{\fill (7.5,2) circle (1pt)};
{\fill (7.5,2.5) circle (1pt)};
{\fill (7.5,3) circle (1pt)};
{\fill (8,0.5) circle (1pt)};
{\fill (8,1) circle (1pt)};
{\fill (8,1.5) circle (1pt)};
{\fill (8,2) circle (1pt)};
{\fill (8,2.5) circle (1pt)};
{\fill (8,3) circle (1pt)};
{\fill (8,3.5) circle (1pt)};
{\fill (8.5,0) circle (1pt)};
{\fill (8.5,0.5) circle (1pt)};
{\fill (8.5,1) circle (1pt)};
{\fill (8.5,1.5) circle (1pt)};
{\fill (8.5,2) circle (1pt)};
{\fill (8.5,2.5) circle (1pt)};
{\fill (8.5,3) circle (1pt)};
{\fill (8.5,3.5) circle (1pt)};
{\fill (9,0) circle (1pt)};
{\fill (9,0.5) circle (1pt)};
{\fill (9,1) circle (1pt)};
{\fill (9,1.5) circle (1pt)};
{\fill (9,2) circle (1pt)};
{\fill (9,2.5) circle (1pt)};
{\fill (9,3) circle (1pt)};
{\fill (9,3.5) circle (1pt)};
{\fill (9,4) circle (1pt)};
{\fill (9.5,0) circle (1pt)};
{\fill (9.5,0.5) circle (1pt)};
{\fill (9.5,1) circle (1pt)};
{\fill (9.5,1.5) circle (1pt)};
{\fill (9.5,2) circle (1pt)};
{\fill (9.5,2.5) circle (1pt)};
{\fill (9.5,3) circle (1pt)};
{\fill (9.5,3.5) circle (1pt)};
{\fill (9.5,4) circle (1pt)};
{\fill (10,0.5) circle (1pt)};
{\fill (10,1) circle (1pt)};
{\fill (10,1.5) circle (1pt)};
{\fill (10,2) circle (1pt)};
{\fill (10,2.5) circle (1pt)};
{\fill (10,3) circle (1pt)};
{\fill (10,3.5) circle (1pt)};
{\fill (10,4) circle (1pt)};
{\fill (10,4.5) circle (1pt)};
{\fill (10.5,0) circle (1pt)};
{\fill (10.5,0.5) circle (2.5pt)};
{\fill (10.5,1) circle (1pt)};
%{\fill (10.5,1.5) circle (1pt)};
{\fill (10.5,2) circle (1pt)};
{\fill (10.5,2.5) circle (1pt)};
{\fill (10.5,3) circle (1pt)};
{\fill (10.5,3.5) circle (1pt)};
{\fill (10.5,4) circle (1pt)};
{\fill (10.5,4.5) circle (2.5pt)};
%{\fill (11,0) circle (1pt)};
%{\fill (11,0.5) circle (1pt)};
%{\fill (11,1) circle (2.5pt)};
%{\fill (11,1.5) circle (1pt)};
%{\fill (11,2) circle (1pt)};
%{\fill (11,2.5) circle (1pt)};
%{\fill (11,3) circle (2.5pt)};
%{\fill (11,3.5) circle (1pt)};
%{\fill (11,4) circle (1pt)};
%{\fill (11,4.5) circle (1pt)};
%{\fill (11,5) circle (2.5pt)};
%%%%%%%%%%%%%%%
{\fill (2.5,0.5) circle (2.5pt)};
{\fill (3,1) circle (2.5pt)};
{\fill (4.5,0.5) circle (2.5pt)};
{\fill (5,1) circle (2.5pt)};
{\fill (5.5,1.5) circle (2.5pt)};
{\fill (6,2) circle (2.5pt)};
{\fill (6.5,2.5) circle (2.5pt)};
{\fill (6.5,0.5) circle (2.5pt)};
{\fill (7,1) circle (2.5pt)};
{\fill (7.5,1.5) circle (2.5pt)};
{\fill (8,2) circle (2.5pt)};
{\fill (8.5,2.5) circle (2.5pt)};
{\fill (9,3) circle (2.5pt)};
{\fill (9.5,3.5) circle (2.5pt)};
{\fill (10,4) circle (2.5pt)};
{\fill (8.5,0.5) circle (2.5pt)};
{\fill (9,1) circle (2.5pt)};
{\fill (9.5,1.5) circle (2.5pt)};
{\fill (10,2) circle (2.5pt)};
{\fill (10.5,2.5) circle (2.5pt)};
%{\fill (11,3) circle (2.5pt)};
{\fill (10.5,0.5) circle (2.5pt)};
%{\fill (11,1) circle (2.5pt)};
}
\end{center}
\caption{The semigroup $M^+_{l,m}$}
\end{figure}
\begin{remark}\label{toric case}
An algorithm for finding a system of generators of $M^+_{l,m}$ is described in \cite{Pa}. 
Consider for example the case when $m=a(q-p)$, that is to say, when $E_{l,m}$ is a toric variety (see Theorem \ref{toric}). 
By applying the algorithm, we see that $M^+_{l,m}$ is minimally generated by 
$(m,0),\; (m+1, 1),\; \dots,\; (aq, ap)$, and  that 
\[
v=(X^m,\; X^{m+1}Y,\; \dots,\; X^{aq}Y^{ap}) 
\in V(m) \oplus \dots \oplus V(aq+ap) \cong 
V(aq) \otimes V(ap).
\]
Note that  $v$ maps to $X^{aq} \otimes Y^{ap} \in V(aq) \otimes V(ap)$ under the above 
isomorphism. 
Moreover, it is known that if $E_{l,m}$ is toric then $E_{l,m}$ is 
isomorphic to the affine cone over the projective embedding of 
$\P$ into $\mathbb P^{(aq+1)(ap+1)-1}$ 
defined by the global sections of $\mathcal O(aq, ap)$: 
\[
\xymatrix{
\sigma: \P \ar@{^{(}-_>}[r] & \mathbb P^{(aq+1)(ap+1)-1}
%([t_0: t_1],\; [s_0: s_1]) \mapsto [t_0^{aq}s_1^{ap}: \dots : t_1^{aq}s_0^{ap}]
}
\]
%\[
%\xymatrix@C=36pt@R=-7pt{
%\sigma: \P \ar@{^{(}-_>}[r] &  \mathbb P^{(aq+1)(ap+1)-1} \\
%\quad \; \; \rotatebox{90}{$\in$} & \rotatebox{90}{$\in$} \\
%\quad \; \; ([t_0: t_1],\; [s_0: s_1]) \ar@{|->}[r] & [t_0^{aq}s_1^{ap}: \dots : t_1^{aq}s_0^{ap}]
%}
%\]
(see \cite[\S 3]{BH} for more details). 
\end{remark}

According to \cite[\S 1]{BH}, 
an affine normal quasihomogeneous $SL(2)$-variety 
$E_{l,m}$ has a description as a categorical quotient 
of a hypersurface in $\C^5$. 
We consider $\C^5$ as the $SL(2)$-module $V(0) \oplus 
V(1) \oplus V(1)$ with coordinates $X_0,\; X_1,\; X_2,\; X_3,\; X_4$, and 
identify $X_1,\; X_2,\; X_3,\; X_4$ with the coefficients of the $2 \times 2$ matrix
\[
\begin{pmatrix}
X_1 & X_3 \\
X_2 & X_4\\
\end{pmatrix}
\]
so that $SL(2)$ acts by left multiplication. 
Let $(l,m) \in \{\mathbb Q \cap (0,1]\} \times \mathbb N$, and write $l=p/q$ as above. We consider actions of the following diagonalizable groups:
\begin{align*}
& G_0:=\left\{\diag(t,\; t^{-p},\; t^{-p},\; t^q,\; t^q)\; : \; t \in \C^*\right\} \cong 
\C^*;\\
& G_m:=\left\{\diag(1,\; \zeta^{-1},\; \zeta^{-1},\; \zeta,\; \zeta)\;  :\; \zeta^m=1\right\} \cong 
\mu_m.
\end{align*} 
It is easy to see that the $SL(2)$-action on $\C^5$ commutes with the $\G$-action. 
\begin{theorem}[{\cite[Theorem 1.6]{BH}}]\label{batyrev}
Let $E_{l,m}$ be a $3$-dimensional affine normal quasihomogeneous $SL(2)$-variety of height 
$l=p/q$ and of degree $m$. Then, $E_{l,m}$ is isomorphic to the categorical quotient 
of the affine hypersurface 
\[
\C^5 \supset H_{q-p}:=(X_0^{q-p}=X_1X_4-X_2X_3)
\]
modulo the action of $\G$. 
\end{theorem}
%%%%%%%%%%%%%%%
% Section 3
%%%%%%%%%%%%%%%%
\section{Hilbert function of the general fibers}\label{s-flat}
%From now on, we always assume that $l=p/q <1$. 
The contents of this section are valid for every affine $SL(2)$-variety $\E$, i.e., valid for both toric and non-toric cases. 

Let 
\[
\xymatrix{
\pi : H_{q-p} \ar[r] &  H_{q-p}\git (\G) \cong E_{l,m}
}
\] 
be the quotient morphism. 
In this section, we determine the flat locus of $\pi$ and the Hilbert function $h:=h_{H_{q-p}}$ of the general fibers of $\pi$. 

Let $x=(1,1,0,0,1) \in H_{q-p}$.  
Then, the $SL(2) \times \G$-orbit of $x$ coincides with 
the open subset $H_{q-p} \cap \{X_0 \neq 0\}$ of $H_{q-p}$, and the categorical 
quotient of $H_{q-p} \cap \{X_0 \neq 0\}$ by $\G$ is isomorphic to the dense 
open orbit $\mathfrak U$ (see the proof of \cite[Theorem 1.6]{BH}).  
Namely, $\mathfrak U$ is the $SL(2)$-orbit of $\pi(x)$. 
\begin{proposition}\label{flat locus}
Keep the above notation.
\begin{itemize}
\item [{\rm(i)}] The quotient morphism $\pi$ is flat over the open orbit $\mathfrak U$.
\item [{\rm(ii)}] For any $g \in SL(2)$, the $\G$-orbit of $g \cdot x$ is closed and isomorphic to $\G$.
\item[{\rm(iii)}] For any $y \in \mathfrak U$, the 
fiber $\pi^{-1}(y)$ is isomorphic to $\G$. 
\end{itemize}
\end{proposition}
\begin{proof}
By the generic flatness theorem, there is a non-empty open subset 
$U$ of $E_{l,m}$ such that $\pi^{-1}(U) \to U$ is flat. 
Since $\mathfrak U \cap U \neq \phi$ and $\pi$ is $SL(2)$-equivariant, it follows that 
$\pi$ is flat over the open orbit $\mathfrak U$.  
To show (ii), 
it suffices to consider the case when $g$ is the 
identity matrix. 
%, since the $SL(2)$-action commutes with the $\G$-action. 
First we see that the $\G$-orbit of $x$ is isomorphic to $\G$, since the stabilizer 
of $x$ is trivial. 
To see that this orbit is closed, consider the ideal 
\[
I_1:=(X_0^{q-p}-X_1X_4,\; X_2,\; X_3,\; 1-X_0^{mp}X_1^{m})
\]
of the polynomial ring $\C[X_0, X_1, X_2, X_3, X_4]$. 
By a simple calculation, we see that the underlying topological space 
of the orbit $(\G) \cdot x$ coincides with the zero set of $I_1$. Therefore, $(\G) \cdot x$ is a closed orbit. Item 
(iii) is a consequence of (ii) and the fact that $\mathfrak U$ is the $SL(2)$-orbit of $\pi(x)$.
\end{proof}
\begin{remark}
We will see in \S 4 that the ideal of the closed orbit  
$(\G) \cdot x=\pi^{-1}(\pi(x)) \subset H_{q-p}$ coincides with $I_1$.
\end{remark} 
\begin{corollary}\label{hilbfunc}
The Hilbert function $h$ of the general fibers of the quotient morphism $\pi$ coincides with that of the regular representation $\C[\G]$: 
\[
h : \Z \to \mathbb Z_{\geq 0},\quad (n,d) \mapsto h(n,d)=1,
\]
where we identify $\Irr(\G)$ with $\Z$. 
\end{corollary}
\begin{remark}
For a $\G$-module $V$, we denote $\Hom^{\G}(M_{(n,d)}, V)$ by $V_{(n,d)}$, where $M_{(n,d)}$ stands for the irreducible representation of weight $(n,d) \in \mathbb Z \times \mathbb Z/m \mathbb Z$. 
\end{remark}
Let us denote by $\mathcal H$ the invariant Hilbert scheme $\H$ associated with the triple $(\G, H_{q-p}, h)$ and consider the Hilbert--Chow morphism 
\[
\xymatrix{
\gamma : \mathcal H \ar[r] & H_{q-p} \git (\G) \cong \E
},
\]
which is an isomorphism over the open orbit $\mathfrak U \subset \E$. 
%To sum up, we get the following $SL(2)$-equivariant commutative diagram:
%\[
%\xymatrix{
%\Univ^{\G}_h(H_{q-p}) \ar[r] \ar[d] & H_{q-p} \ar[d]^{\pi} \\
%\mathcal H:=\H \ar[r]^{\gamma} & H_{q-p} \git (\G) \cong E_{l,m}
%}
%\]
%The Hilbert--Chow morphism $\gamma$ is an isomorphism over the open orbit $\mathfrak U$, and its restriction to the main component 
%\[
%\gamma{ |_{\mathcal H^{main}}} : \mathcal H^{main}=
%\overline{\gamma^{-1}(\mathfrak U)} \to E_{l,m}
%\]
%is projective and birational. 
%%%%%%%%%% 
% Section 4
%%%%%%%%%%%%%%%%%%%%%%%%%%
\section{Calculation of ideals}\label{s-ideal}
%Let $A$ be the polynomial ring $\C[X_0, X_1, X_2, X_3, X_4]$, and consider the following ideals of $A$ (the ideal $I_1$ has already appeared in the proof of Proposition \ref{flat locus}):
%\begin{align*}
%& I_{1}:=(X_0^{q-p}-X_1X_4,\; X_2,\; X_3,\; 1-X_0^{mp}X_1^{m}); \\
%& I_{0}:=(X_0^{q-p}-X_1X_4,\; X_2,\; X_3,\; X_0^{mp}X_1^{m}).
%\end{align*}
%Also, let 
%\begin{align*}
%& J_{1}:=(X_0^{q-p},\; X_2,\; X_4,\; 1-X_1^{aq}X_3^{ap}); \\
%& J_{0}:=(X_0^{q-p},\; X_2,\; X_4,\; X_1^{aq}X_3^{ap}).
%\end{align*}
For each $s \in \C$, we consider two kinds of ideals
\[
I_{s}:=(X_0^{q-p}-X_1X_4,\; X_2,\; X_3,\; s-X_0^{mp}X_1^{m})
\]
and 
\[
J_{s}:=(X_0^{q-p},\; X_2,\; X_4,\; s-X_1^{aq}X_3^{ap})
\]
of the polynomial ring $A:=\C[X_0,X_1, X_2, X_3, X_4]$. 
We remark that the ideal $I_1$ has already appeared in the proof of Proposition \ref{flat locus}.  
%The main results of this section are Theorems \ref{idealU} and ref{idealD}. 
%They play important roles in determining the main component 
%$\mathcal H^{main}$, which will be shown to coincide with $\mathcal H$ (Corollary \ref{toric sm}). 
\begin{theorem}\label{idealU}
The quotient ring $A/I_s$ has Hilbert function $h$ for any $s \in \C$, i.e., we have $\dim (A/I_s)_{(n,d)}=h(n,d)$ for any weight $(n,d) \in \Z$. 
\end{theorem}
\begin{theorem}\label{idealD}
If $\E$ is toric, then the quotient ring $A/J_s$ 
has Hilbert function $h$ for any $s \in \C$, i.e., we have $\dim (A/J_s)_{(n,d)}=h(n,d)$ for any weight $(n,d) \in \Z$.
\end{theorem}
\begin{remark}\label{s1}
If $s \in \C^*$, then we see that $I_s$ and $J_s$ are  $SL(2)$-translates of $I_1$ and $J_1$, respectively. 
\end{remark}
%Moreover, we will see as a consequence of Theorem \ref{idealD} that the invariant Hilbert scheme $\mathcal H$ contains a unique fixed point for the action of the Borel subgroup $B$ (Proposition \ref{candidate}). 
%Furthermore, we will see in Lemma \ref{bij} that any closed point of $\mathcal H$ 
%is an $SL(2)$-translate of $[I_s]$ or $[J_s]$.  
In what follows, we prepare some lemmas needed for the proof of Theorems \ref{idealU} and \ref{idealD}.  
Let $S$ be the coordinate ring of $H_{q-p}$: 
\[
S=\C[H_{q-p}] \cong A/(X_0^{q-p}-X_1X_4+X_2X_3).
\] 
%Let $[Z]$ be any point in $\H$. 
%Then we have 
%\begin{equation}
%\dim (F_{(n,d)}/F_{(n,d)} \cap I_Z)=h(n,d)=1, \label{dim}
%\end{equation}
%where $I_Z$ stands for the ideal of $Z \subset H_{q-p}$. 
\begin{lemma}\label{generator}
With the above notation, we have the following. 
\begin{itemize}
\item [{\rm(i)}] $S_{(-p, -1)}=S^{\G}X_1 + S^{\G}X_2$.
\item [{\rm(ii)}] $S_{(q,1)}=S^{\G}X_3 + S^{\G}X_4$.
\end{itemize}
\end{lemma}
\begin{proof}
Since 
$X_1, X_2 \in S_{(-p, -1)}$, it is clear that 
$S_{(-p, -1)} \supset S^{\G} X_1 + S^{\G} X_2$. 
To see the other inclusion, take an arbitrary $f=X_0^{d_0}X_1^{d_1}X_2^{d_2}X_3^{d_3}X_4^{d_4} \in A_{(-p, -1)}$. 
If either $d_1 > 0$ or $d_2 > 0$ holds, then we clearly have $f \in A^{\G}X_1+A^{\G}X_2$. 
Otherwise, $f$ is of the form $f=X_0^{d_0}X_3^{d_3}X_4^{d_4}$. 
But this contradicts to $f \in A_{(-p,-1)}$, since the $G_0$-weights of $X_0,\; X_3$, and $X_4$ are all 
positive. This shows (i). 
Item (ii) follows in a similar way.  
\end{proof}
\begin{remark}\label{invariant ring}
Let $i \in \{1,2\}$ and $j \in \{3,4\}$. Then the invariant ring $\C[X_0, X_{i}, X_{j}]^{\G}$ is given as follows (see the proof of \cite[Theorem 1.6]{BH}): 
\[
\C[X_0^{pu_1-qu_2}X_{i}^{u_1}X_{j}^{u_2}\; :\; (u_1, u_2) \in M^+_{l,m}].
\]
If $\E$ is toric, then we see by Remark \ref{toric case} that this coincides with 
\[
\C[
X_0^{(ap-u)(q-p)}X_{i}^{m+u}X_{j}^u\; :\; 0 \leq u \leq ap]. 
\] 
\end{remark}
Let $(n,d) \in \mathbb Z \times \mathbb Z/ m \mathbb Z$.  
As we have seen in \S \ref{s-hilb}, there is a finite-dimensional $SL(2)$-module 
$F_{n,d}$ that generates the weight space $S_{(n,d)}$ over the invariant ring $S^{\G}$. 
By Lemma \ref{generator}, we can take $F_{-p, -1}
=\langle X_1, X_2\rangle$ and 
$F_{q, 1}
=\langle X_3, X_4\rangle$. 
Also, it follows that for any closed point $[I] \in \mathcal H$, 
%, and $\mathcal K$ the kernel of the natural map (which has already been considered in \S \ref{s-hilb} in a general setting)
%\[
%F_{-p,-1} \hookrightarrow S_{(-p,-1)} \to (S/I)_{(-p,-1)}.
%\] 
%Then, we have $\dim F_{-p,-1}/\mathcal K=h(-p,-1)=1$, and hence 
we have 
\begin{equation}
s_1X_1+s_2X_2 \in I \label{s1s2}
\end{equation}
and 
\begin{equation}
s_3X_3+s_4X_4 \in I \label{s3s4}
\end{equation}
for some $(s_1, s_2) \neq 0$ and $(s_3, s_4 ) \neq 0$, respectively,  since $h(-p,-1)=h(q,1)=1$. 
\begin{lemma}\label{boundarybasis}
We have 
$\dim (A/I_1)_{(n,d)} \geq h(n,d)$ and $\dim (A/I_0)_{(n,d)} \geq h(n,d)$ for any 
$(n,d) \in \mathbb Z \times \mathbb Z/m \mathbb Z$. 
\end{lemma}
\begin{proof}
We have seen in \S \ref{s-flat} that the open orbit $\mathfrak U \subset \E$ coincides with the $SL(2)$-orbit of $\pi(x)$, where $x=(1,1,0,0,1) \in H_{q-p}$. 
Let $[I] \in \gamma^{-1}(\mathfrak U)$ be a point such that $\gamma([I])=\pi(x)$. 
Since $X_0^{mp}X_1^m$, $X_0^{mp}X_2^m \in A^{\G}$, we heve $1-X_0^{mp}X_1^m$, $X_0^{mp}X_2^m \in I$. 
%, concerning the $X_0$, $X_1$, $X_2$-coordinates of $x$. 
Therefore, taking \eqref{s1s2} into account, we get $X_2 \in I$. 
Similarly, since we have $X_1^{aq}X_3^{ap}$, $X_1^{aq}X_4^{ap} \in A^{\G}$, it follows that   $X_1^{aq}X_3^{ap}$,  $1-X_1^{aq}X_4^{ap} \in I$, and this implies that $X_3 \in I$ concerning  \eqref{s3s4}.  
Therefore, we have $I_1 \subset I$,  and hence the natural surjection $A/I_1 \to A/I$. 
It follows that $\dim (A/I_1)_{(n,d)} \geq \dim (A/I)_{(n,d)}=h(n,d)$. 
%In view of Remark \ref{s1}, this implies that $\dim (A/I_s)_{(n,d)} \geq 1$ holds for any $s \in \C^*$. 
Next, let $[I'] \in \gamma^{-1}(O)$ be a point such that 
$\gamma([I']) \in H_{q-p} \cap \{X_2=X_3=0\} \git (\G)$.  
We see in a similar way that $I_0 \subset I'$ holds, and therefore $\dim (A/I_0)_{(n,d)} \geq h(n,d)$. 
\end{proof}
\begin{lemma}\label{boundarybasis2}
We have 
$\dim (A/J_1)_{(n,d)} \geq h(n,d)$ and $\dim (A/J_0)_{(n,d)} \geq h(n,d)$ for any 
$(n,d) \in \mathbb Z \times \mathbb Z/m \mathbb Z$. 
\end{lemma}
\begin{proof}
We can easily see that the $2$-dimensional orbit $\mathfrak D \subset \E$ coincides with the  $SL(2)$-orbit of $\pi(x')$, where $x'=(0,1,0,1,0) \in H_{q-p}$.  
Let $[J] \in \gamma^{-1}(\mathfrak D)$ be a point such that $\gamma([J])=\pi(x')$. 
Then, as in the proof of Lemma \ref{boundarybasis}, we can show that $J_1 \subset J$, and therefore $\dim (A/J_1)_{(n,d)} \geq \dim (A/J)_{(n,d)}
=h(n,d)$. 
Next, let $[J']\in \gamma^{-1}(O)$ be a point such that 
$\gamma([J'])\in H_{q-p} \cap \{X_2=X_4=0\} \git (\G)$. 
We see by following the same line that $J_0 \subset J'$, and thus $\dim (A/J_0)_{(n,d)} \geq h(n,d)$. 
\end{proof}
Let $j \in \{3,4\}$, and set $R:=\C[X_0, X_1, X_j]$. 
For each $c,n \in \mathbb Z$,  consider the vector subspaces 
\[
R^c:=\langle X_0^{d_0}X_1^{d_1}X_{j}^{d_{j}} \in R\; :\; d_1-d_{j}=c\rangle
\]
and
\[
R_n:=\langle X_0^{d_0}X_1^{d_1}X_{j}^{d_{j}} \in R\; :\; d_0-pd_1+qd_{j}=n\rangle
\] 
of $R$. 
Then we see that the polynomial ring $R$ decomposes as follows: 
\begin{equation*}
R=\bigoplus_{c \in \mathbb Z} R^c=\bigoplus_{n \in \mathbb Z} R_n.
\end{equation*}
Moreover, let $R^c_n:=R^c \cap R_n$. Then, for any weight $(n,d) \in \mathbb Z \times \mathbb Z/m \mathbb Z$, we have
\begin{equation*}
R_{(n,d)}=\bigoplus_{c \equiv d\; (mod\; m)} R^c_n.\label{new}
\end{equation*}
\begin{lemma}\label{minc}
For any $(n,d) \in \mathbb Z \times \mathbb Z/m \mathbb Z$, the minimum
\[
c_{(n,d)}:=\min\{c \in \mathbb Z\; :\; c \equiv d\; (mod\; m),\; R^c_n \neq 0\}
\]
exists.
\end{lemma}
\begin{proof}
Take an arbitrary $0 \neq X_0^{d_0}X_1^{d_1}X_{j}^{d_{j}}  \in R^c_n$. Then we have 
$n=d_0-pd_1+qd_{j}=d_0+(q-p)d_1-qc$, and hence $c \geq -n/q$. 
\end{proof}
\begin{example}
If $0 \leq n \leq q-p$, then we have $c_{(n,0)}=0$.  
Indeed, suppose that $R^c_n \neq 0$ for some $c<0$. 
Then we can take a non-zero  
$X_0^{d_0}X_1^{d_1}X_{j}^{d_{j}} \in R^{c}_n$, and we have 
$n=d_0+(q-p)d_1-qc \geq q > q-p$. 
By a direct calculation, we see that $R^0_n=\langle X_0^n\rangle$ if 
$0 \leq n <q-p$, and that $R^0_{q-p}=\langle X_0^{q-p},\; X_1X_{j}\rangle$. 
\end{example}
Consider a $\mathbb Z$-linear map $\mu : \mathbb Z^3 \to \mathbb Z^3$ defined by
\[
(d_0, d_1, d_{j}) \mapsto \mu(d_0, d_1, d_{j}):=(d_0-pd_1+qd_{j},\; d_1-d_{j},\; pd_1-qd_{j}).
\]
We see that $\mu$ is injective. Let us denote by $\Lambda$ the image of $\mu|_{\mathbb Z_{\geq 0}^3}$, and define 
\begin{equation*}
R_{\lambda}:=\langle X_0^{d_0}X_1^{d_1}X_{j}^{d_{j}} \in R\; :\; \mu(d_0, d_1, d_{j})=\lambda \rangle
\end{equation*}
for each $\lambda \in \Lambda$. 
Then we have
\begin{equation*}
R=\bigoplus_{\lambda \in \Lambda} R_{\lambda}. 
\end{equation*}
The next lemma follows from the definition of $\mu$ by a direct calculation. 
\begin{lemma}\label{1dim}
The following properties are true. 
\begin{itemize}
\item [\rm(i)] Let $\lambda=(n,c,\w) \in \Lambda$. Then, the vector space $R_{\lambda}$ is spanned by 
\begin{equation*}
f_{\lambda}:=X_0^{n+\w}X_1^{\frac{qc-\w}{q-p}}X_{j}^{\frac{pc -\w}{q-p}}. \label{function}
\end{equation*}
\item [\rm(ii)] For any $\lambda,\; \lambda' \in \Lambda$, we have 
$f_{\lambda} f_{\lambda'}=f_{\lambda+\lambda'}$. 
\end{itemize}
\end{lemma}
\begin{example}
Let $(u_1, u_2) \in M^+_{l,m}$, and $\lambda=(0,u_1-u_2, pu_1-qu_2)$. Then we have $f_{\lambda}=X_0^{pu_1-qu_2}X_1^{u_1}X_j^{u_2} \in R^{u_1-u_2}_0$.  
\end{example}
\begin{remark}
The polynomial ring $R$ has a natural $\mathbb Z \times \mathbb Z/m \mathbb Z$-grading defined by the $\G$-action, but each graded component $R_{(n,d)}$ with respect to this grading  is infinite-dimensional. 
Lemma \ref{1dim} implies that $R$ admits another grading, namely, the $\Lambda$-grading, such that each graded component $R_{\lambda}$ is one-dimensional. 
We will see below that this makes it easier to analyze the structure of the weight space $R_{(n,d)}$. 
\end{remark}
Next, consider the projection $\tilde{\mu} : \mathbb Z^3 \to \mathbb Z^2,\; (n,c,\w) \mapsto (n,c)$. 
Let $\mu'=\tilde{\mu} \circ \mu$, and 
denote by $\Lambda'$ the image of $\mu'|_{\mathbb Z_{\geq 0}^3}$. Then we have
\begin{equation*}
R=\bigoplus_{(n,c) \in \Lambda'} R^c_n, \quad 
R^c_n=\bigoplus_{\lambda \in \tilde{\mu}^{-1}(n,c)} R_{\lambda}. \label{cn}
\end{equation*}
%\begin{remark}
%Let $(n,c) \in \Lambda'$. If $n \leq 0$, then we have $\w \geq 0$ for any $(n,c,\w) \in \tilde{\mu}^{-1}(n,c)$ since $n+\w \geq 0$. 
%\end{remark}
As an immediate consequence of Lemma \ref{1dim}, we get the following
\begin{lemma}\label{q-pz}
Let $(n,c) \in \Lambda'$. Then we have $\w-\w' \in (q-p)\mathbb Z$ for any $(n,c,\w)$, $(n,c,\w') \in \tilde{\mu}^{-1}(n,c)$. 
\end{lemma}
Concerning Lemma \ref{1dim}, we also see that the minimum 
\begin{align*}
& \w_{(n,c)}:=\min\{\w \in \mathbb Z\; :\; (n,c,\w) \in \tilde{\mu}^{-1}(n,c)\}
\end{align*}
exists for any $(n,c) \in \Lambda'$. 
\begin{remark}
The maximum $\max\{\w \in \mathbb Z\;:\; (n,c,\w) \in \tilde{\mu}^{-1}(n,c)\}$ also exists, and hence the vector space $R^c_n$ is finite-dimensional.
\end{remark}
\begin{lemma}\label{min}
Let $(n,c,\w) \in \Lambda$. Then, we have $n+\w < q-p$ if and only if $\w=\w_{(n,c)}$. 
\end{lemma}
\begin{proof}
First, notice that the condition $(n,c,\w) \in \Lambda$ implies that $(n,c) \in \Lambda'$ and that $(n,c,\w) \in \tilde{\mu}^{-1}(n,c)$. 
Suppose that $n+\w \geq q-p$, and set $\w '=\w-(q-p)$. 
Let $d_0=n+\w'$, $d_1=\frac{qc-\w}{q-p}+1$, and $d_j=\frac{pc-\w}{q-p}+1$. Then, taking Lemma \ref{1dim} into account,  we see that $(d_0,d_1, d_j) \in \mathbb Z_{\geq 0}^3$. By a direct calculation, we have $\mu(d_0, d_1, d_j)=(n, c, \w')$, and therefore $(n, c, \w') \in \Lambda$. 
It follows that $\w' \geq \w_{(n,c)}$. 
Conversely, suppose that $\w > \w_{(n,c)}$. Then we have $\w-\w_{(n,c)} \geq q-p$ by Lemma \ref{q-pz}. 
Since we have $n+\w_{(n,c)} \geq 0$ by Lemma \ref{1dim}, it follows that $n+\w \geq n+\w_{(n,c)}+q-p \geq q-p$. 
\end{proof}
\begin{definition}
For each $(n,d) \in \mathbb Z \times \mathbb Z/m \mathbb Z$, we define:
\begin{itemize}
\item [\rm(i)] $\Lambda_{(n,d)}:=\{(n,c,\w) \in \Lambda \; :\; c \equiv d\; (mod\; m)\}$;
\item [\rm(ii)] $\lambda_{(n,d)}:=(n,c_{(n,d)}, \w_{(n,c_{(n,d)})}) \in \Lambda_{(n,d)}$. 
\end{itemize}
\end{definition}
%Notice that $\Lambda_{(n,d)}$ coincides with 
%$\{(n,c,\w) \in \tilde{\mu}^{-1}(n,c)\; :\; (n,c) \in \Lambda',\; c \equiv d\; (mod\; m)\}$.
%Therefore, one obtains
Using the notation introduced above, we have different ways of 
expressing the weight space $R_{(n,d)}$: 
\begin{equation*}
R_{(n,d)}=\bigoplus_{\substack{
c \equiv d\; (mod\; m) \\
c \geq c_{(n,d)}}} R^c_n  
 =\bigoplus_{\substack{
c \equiv d\; (mod\; m)\\
c \geq c_{(n,d)}}}\left(\bigoplus_{\lambda \in \tilde{\mu}^{-1}(n,c)} R_{\lambda}\right) 
 =\bigoplus_{\lambda \in \Lambda_{(n,d)}} R_{\lambda}. \label{relation1}
\end{equation*}
\begin{example}
Let $l=p/q=1/3$, and $m=2$. Then by Remark \ref{toric case} the semigroup $M^+_{\frac{1}{3}, 2}$ is minimally generated by $(2,0)$ and $(3,1)$. 
Therefore, in view of Remark \ref{invariant ring}, we have $R_{(0,0)}=R^{\G}=\C[X_0^2X_1^2, X_1^3X_j]$.  We can also calculate the following: 
\begin{align*}
& R^0_0=\C; \\ 
& R^2_0=R_{(0,2,0)} \oplus R_{(0,2,2)},\; f_{(0,2,0)}=X_1^3X_j,\; f_{(0,2,2)}=X_0^2X_1^2; \\  
& R^0_1=R_{(1,0,0)},\;  f_{(1,0,0)}=X_0; \\  
& R^2_1=R_{(1,2,0)} \oplus R_{(1,2,2)},\: f_{(1,2,0)}=X_0X_1^3X_j,\; f_{(1,2,2)}=X_0^3X_1^2; \\  
& R^0_2=R_{(2,0,-2)} \oplus R_{(2,0,0)},\; f_{(2,0,-2)}=X_1X_j,\; f_{(2,0,0)}=X_0^2; \\ 
& R^2_2=R_{(2,2,-2)} \oplus R_{(2,2,0)} \oplus R_{(2,2,2)},\;  f_{(2,2,-2)}=X_1^4X_j^2,\;  f_{(2,2,0)}=X_0^2X_1^3X_j,\;  f_{(2,2,2)}=X_0^4X_1^2. 
\end{align*}
We see that $\lambda_{(0,0)}=(0,0,0)$, $\lambda_{(1,0)}=(1,0,0)$, and $\lambda_{(2,0)}=(2,0,-2)$. 
\end{example}
%%%%%%%%%%%5
\begin{lemma}\label{cont6}
Let $\lambda=(n,c,\w),\; \lambda'=(n,c', \w') \in \Lambda_{(n,d)}$. Then we have the following. 
\begin{itemize}
\item [\rm(i)] If $c=c'$, then we have $f_{\lambda}-f_{\lambda'} \in
(X_0^{q-p}-X_1X_{j})$.  
\item [\rm(ii)] If $c >c_{(n,d)}$, then we have $f_{\lambda} \in (X_0^{q-p}-X_1X_{j},\; X_0^{mp}X_1^m)$. 
\item [\rm(iii)] We have $f_{\lambda}-f_{\lambda'} \in (X_0^{q-p}-X_1X_{j},\; 1-X_0^{mp}X_1^m)$. 
\end{itemize}
If $\E$ is toric, i.e., if we have $m=a(q-p)$, then 
the following properties are true. 
\begin{itemize}
\item [\rm(iv)] If $\w=\w'$, then we have $f_{\lambda}-f_{\lambda'} \in
(1-X_1^{aq}X_{j}^{ap})$. 
\item [\rm(v)] If $\w=\w_{(n,c)}$ and $\w'=\w_{(n,c')}$, then $\w=\w'$. In particular, we have $f_{\lambda}-f_{\lambda'} \in (1-X_1^{aq}X_{j}^{ap})$.
\end{itemize}
\end{lemma}
\begin{proof}
By Lemma \ref{1dim}, we have 
$f_{\lambda}=X_0^{n+\w}X_1^{\frac{qc-\w}{q-p}}X_j^{\frac{pc-\w}{q-p}}$ and $f_{\lambda'}=X_0^{n+\w'}X_1^{\frac{qc'-\w'}{q-p}}X_j^{\frac{pc'-\w'}{q-p}}$.   
We can write $c$ and $c'$ as $c=c_{(n,d)}+mx$ and $c'=c_{(n,d)}+mx'$ with some $x,\; x' \in \mathbb Z_{\geq 0}$, respectively.  
Without loss of generality, we may assume that $c \geq c'$.
 
(i) We may assume that $\w \geq \w'$. 
Then, by Lemma \ref{q-pz}, we have $\w-\w'=y(q-p)$ for some $y \geq 0$. 
Therefore, we have
\[
f_{\lambda}-f_{\lambda'}=
X_0^{n+\w'}X_1^{\frac{qc-\w}{q-p}}X_{j}^{\frac{pc-\w}{q-p}}\{(X_0^{q-p})^y-(X_1X_{j})^y\} \in (X_0^{q-p}-X_1X_{j}). 
\]

(ii) We first remark that $f_{(0,m,mp)}=X_0^{mp}X_1^m$. 
Let $\lambda''=\lambda_{(n,d)}+x(0, m, mp)=(n,c,\w_{(n,c_{(n,d)})}+mpx)$. 
Then we have $f_{\lambda}-f_{\lambda''} \in (X_0^{q-p}-X_1X_{j})$ by (i). 
Since we have $f_{\lambda''}=f_{\lambda_{(n,d)}}(X_0^{mp}X_1^m)^x$ by Lemma \ref{1dim}, it follows that $f_{\lambda} \in (X_0^{q-p}-X_1X_{j},\; X_0^{mp}X_1^m)$.

(iii) Taking (i) into account, we may assume that $c > c'$. 
Let $\lambda''$ be as in the proof of (ii), and $\lambda'''=\lambda_{(n,d)}+x'(0, m, mp)=(n,c',\w_{(n,c_{(n,d)})}+mpx')$.  
Then we have 
$f_{\lambda''}-f_{\lambda'''}=f_{\lambda_{(n,d)}}(X_0^{mp}X_1^m)^{x'}\{(X_0^{mp}X_1^m)^{x-x'}-1\} \in (1-X_0^{mp}X_1^m)$. Therefore we get 
\[
f_{\lambda}-f_{\lambda'}=(f_{\lambda}-f_{\lambda''})+(f_{\lambda''}-f_{\lambda'''})+(f_{\lambda'''}-f_{\lambda'}) \in (X_0^{q-p}-X_1X_{j},\; 1-X_0^{mp}X_1^m),
\]
since we have $f_{\lambda}-f_{\lambda''}$, $f_{\lambda'''}-f_{\lambda'} \in (X_0^{q-p}-X_1X_{j})$ by (i). 

(iv) We get  
$f_{\lambda}-f_{\lambda'}=\{(X_1^{aq}X_{j}^{ap})^{x-x'}-1\}f_{\lambda'} \in (1-X_1^{aq}X_{j}^{ap})$ by a direct calculation using $m=a(q-p)$. 

(v) Set $d_0=n+\w_{(n,c')}$, $d_1=\frac{qc'-\w_{(n,c')}}{q-p}$, and $d_j=\frac{pc'-\w_{(n,c')}}{q-p}$. Then we see that  $d_0,\; d_1,\; d_j \in \mathbb Z_{\geq 0}$ and that $\mu(d_0,d_1, d_j)=(n,c', \w_{(n,c')})$.  
By the definition of $\mu$, we have $\mu(d_0, d_1+aq(x-x'), ap(x-x'))=(n,c,\w_{(n,c')})$.  Therefore $(n,c,\w_{(n,c')}) \in \mu'^{-1}(n,c)$, and we have $\w_{(n,c')}=\w_{(n,c)}+y(q-p)$ for some $y \geq 0$ by Lemma \ref{q-pz} and by the minimality of $\w_{(n,c)}$. Further, we have $q-p > n+\w_{(n,c')}=n+\w_{(n,c)}+y(q-p) \geq y(q-p)$ by Lemma \ref{min}, and thus we get $y=0$.  
\end{proof}
%\begin{remark}
%In Lemma \ref{cont6}, items (i), (ii), and (iii) are valid without the toric hypothesis. 
%\end{remark}
\begin{lemma}\label{cont8}
Let $\lambda=(n,c,\w) \in \Lambda_{(n,d)}$. Then the following properties are true. 
\begin{itemize}
\item [\rm(i)] We have $f_{\lambda} \in (X_0^{q-p})$ if and only if $\w > \w_{(n,c)}$. 
\item [\rm(ii)] Suppose that $\E$ is toric, i.e., that we have $m=a(q-p)$.  If $\w=\w_{(n,c)}$ and $c > c_{(n,d)}$, then we have $f_{\lambda} \in (X_1^{aq}X_{j}^{ap})$. 
\end{itemize}
\end{lemma}
\begin{proof}
First, we have  $f_{\lambda}=X_0^{n+\w}X_1^{\frac{qc-\w}{q-p}}X_j^{\frac{pc-\w}{q-p}}$. 
Item (i) is an immediate consequence of Lemma \ref{min}. 
Suppose that $\E$ is toric. 
If $\w=\w_{(n,c)}$, then we have $\w=\w_{(n,c_{(n,d)})}$ by Lemma \ref{cont6} (v). The condition $c >c_{(n,d)}$ implies that we can write $c=c_{(n,d)}+mx$ for some $x>0$, and therefore we have $f_{\lambda}=f_{\lambda_{(n,d)}}(X_1^{aq}X_{j}^{ap})^x \in (X_1^{aq}X_j^{ap})$. 
\end{proof}
\begin{proof}[Proof of Theorem \ref{idealU}]
Taking Remark \ref{s1} into account, it suffices to prove the theorem for $s=0, 1$. Let $j=4$, i.e., $R=\C[X_0, X_1, X_4]$, and set $\widetilde{I_0}=(X_0^{q-p}-X_1X_4,\; X_0^{mp}X_1^{m})$ and $\widetilde{I_1}=(X_0^{q-p}-X_1X_4,\; 1-X_0^{mp}X_1^{m})$. 
%Then, we have 
%$A/I_1
%\cong 
%R/\widetilde{I_1}$, $A/I_0 \cong R/\widetilde{I_0}$. 
In view of Lemma \ref{boundarybasis}, it suffices to show that  
$\dim (R/\widetilde{I_0})_{(n,d)} \leq h(n,d)=1$ and $\dim (R/\widetilde{I_1})_{(n,d)} \leq h(n,d)=1$ 
hold for any weight $(n,d) \in \mathbb Z \times \mathbb Z /m \mathbb Z$. 
We see that the weight space $R_{(n,d)}$ decomposes as 
\[
R_{(n,d)}=R^{c_{(n,d)}}_n \oplus R_{(n,d)}', 
\]
where $R_{(n,d)}'=\bigoplus_{\substack{
c \equiv d\; (mod\; m)\\
c > c_{(n,d)}}} R^c_n$. 
%Taking Remark \ref{s1} into account, it is enough to  consider the cases when $s=0, 1$. 
%If $s=0$,   
First, we have 
$R_{(n,d)}' \subset \widetilde{I_0}$ by Lemma \ref{cont6} (ii).
Therefore, we get $\dim (R/\widetilde{I_0})_{(n,d)} \leq 1$ by applying Lemma \ref{cont6} (i) with $c=c_{(n,d)}$.  
Similarly, we have $\dim (R/\widetilde{I_1})_{(n,d)} \leq  1$ by Lemma \ref{cont6} (iii).  
\end{proof}
\begin{proof}[Proof of Theorem \ref{idealD}]
As in the proof of Theorem \ref{idealU}, we only consider the cases where $s=0, 1$. 
Let $j=3$, i.e., $R=\C[X_0, X_1, X_3]$, and set $\widetilde{J_0}=(X_0^{q-p},\; X_1^{aq}X_3^{ap})$ and $\widetilde{J_1}=(X_0^{q-p},\; 1-X_1^{aq}X_3^{ap})$. 
Since $R_{\lambda_{(n,d)}}$ is $1$-dimensional, we have $\dim (R/\widetilde{J_0})_{(n,d)} \leq 1$ by Lemma \ref{cont8}, and hence the equality $\dim (R/\widetilde{J_0})_{(n,d)} =1$ concerning Lemma \ref{boundarybasis2}. 
Also, $\dim (R/\widetilde{J_1})_{(n,d)} =1$ follows from Lemmas \ref{boundarybasis2}, \ref{cont6} (v), and \ref{cont8} (i). 
\end{proof}
%%%%%         これは消さない!
%\begin{remark}\label{remark}
%Notice that we does not use the toric hypothesis for the proof of Theorem \ref{idealU}, while we does for that of  
%Theorem \ref{idealD}.  However, Theorem \ref{idealD} can be generalized to the non-toric case. 
%Namely, we can show that the following ideals of $A$ have Hilbert function $h$ for every $s \in \C$: 
% (see \cite[Corollary 10.15]{Ku}):
%$(X_0^k,\; X_2,\; X_4,\; s-X_1^{aq}X_3^{ap})$.  
%The proof follow that same line as that of Theorem \ref{idealD}, but the argument is more intricate since the structure of $R=\C[X_0, X_1, X_3]$ as a $\Lambda$-graded $\Lambda$ ring becomes much more complicated. 
%\end{remark}
%%%%%%      ここまで消さない!
\begin{corollary}\label{isom}
The $SL(2)$-equivariant isomorphism 
$\gamma|_{\gamma^{-1}(\mathfrak U)} : \gamma^{-1}(\mathfrak U) \to \mathfrak U$ is 
given by sending $[I_1]$ to $\pi(x)$. 
%, where $[I_1]$ stands for the closed point of $\H$ corresponds to the ideal $I_1$. 
\end{corollary}
\begin{proof}
Taking the proof of Proposition \ref{flat locus} (ii) into account, it follows from Theorem \ref{idealU} that the defining ideal of the closed orbit $\pi^{-1}(\pi(x)) \cong \G$ is $I_1$, since $h$ is the Hilbert function of the regular representation $\C[\G]$. 
\end{proof}
%%%%%%%%
\section{Borel-fixed points}\label{s-borel}
Throughout this and next section we assume that an  affine $SL(2)$-variety $\E$ is toric. 
In this section, we consider the action of the Borel subgroup $B \subset SL(2)$ on $\mathcal H$ induced by the $SL(2)$-action on $H_{q-p}$. 
Let us denote by $\mathcal H^B$ the set of $B$-fixed points. 
\begin{proposition}\label{candidate}
We have $\mathcal H^B=\{[J_0]\}$. 
\end{proposition}
\begin{proof}
%Recall that for any 
%$g=
%\begin{pmatrix}
%t & u \\
%0 & t^{-1}
%\end{pmatrix}
%\in B
%$, its action on the variables is given by 
%\[
%g \cdot \left(X_0,\; 
%\begin{pmatrix}
%X_1 & X_3 \\
%X_2 & X_4 
%\end{pmatrix}
%\right)
%=
%\left(X_0,\; 
%\begin{pmatrix}
%tX_1+uX_2 & tX_3+uX_4 \\
%t^{-1}X_2 & t^{-1}X_4
%\end{pmatrix}
%\right).
%\]
Take an arbitrary $[J] \in \mathcal H^B$. 
By \eqref{s1s2}, we have $s_1X_1+s_2X_2 \in J$ for some $(s_1, s_2) \neq 0$. Since $J$ is stable under the $B$-action, we see that $X_2 \in J$. 
By the same argument using \eqref{s3s4}, we have $X_4 \in J$,  and hence $(X_0^{q-p}, X_2, X_4) \subset J$.  
Moreover, since $[J]$ is fixed by the action of $B$, it follows that $\gamma([J])=O \in \E$, and hence we have $X_1^{aq}X_3^{ap} \in J$ in view of  Remark \ref{invariant ring}. 
Therefore, $J_0 \subset J$. 
We have seen in Theorem \ref{idealD} that $J_0$ has Hilbert function $h$, and 
thus we get $J_0 = J$. 
\end{proof}
%%%%%%%%%%        これは消さない!!!!!!
%\begin{remark}\label{non-toric}
%If $E_{l,m}$ is non-toric, then the invariant Hilbert scheme $\mathcal H$ contains more than one Borel-fixed points. %(\cite[Corollary 11.1]{Ku}).  
%Here we describe the corresponding ideals. 
%First, let 
%$K$ be the ideal of $R=\C[X_0, X_1, X_3]$ generated by non-constant invariants, i.e., $K=(X_0^{pu_1-qu_2}X_1^{u_1}X_3^{u_2})_{(u_1,u_2) \in M^+_{l,m} \setminus \{(0,0)\}}$. 
%Next, let $\alpha$ and $\beta$ be the quotient and the remainder of $mp$ divided by $q-p$, respectively, i.e., $mp=\alpha (q-p)+\beta$. Set $t=\frac{q-p-\beta}{k}$, and consider the Hirzebruch--Jung continued fraction of $b/t$:
%\begin{equation*}
%\frac{b}{t}=
%c_1 - \cfrac{1}{c_2 -
   %      \cfrac{1}{\ldots -
      %   \cfrac{1}{c_r}}}
%\end{equation*}
%Let 
%$P_0=0$, $Q_0=-1$, 
%$P_1=1$, and $Q_1=0$. 
%For $2 \leq i \leq r+1$, we recursively define 
%$P_i=c_{i-1} P_{i-1}-P_{i-2}$ and $Q_i=c_{i-1} Q_{i-1}-Q_{i-2}$.   
%Using these, set $n_i=k(tP_i-bQ_i)$, and consider the ideal 
%\[
%J^i_0=(X_0^{n_{i-1}},\; X_2,\; X_4,\; X_1^{(\alpha+1+m)P_i-Q_i}X_3^{(\alpha+1)P_i-Q_i})+K \subset A
%\]
%for each $1 \leq i \leq r+1$ (
%we note that $J^{r+1}_0$ coincides with the ideal $(X_0^k,\; X_2,\; X_4,\; X_1^{aq}X_3^{ap})$ defined in Remark \ref{remark}).   
%If $\E$ is non-toric, then the set of Borel-fixed points is given as follows: 
%$\mathcal H^{B}=\{[J^1_0],\; \dots,\; [J^{r+1}_0]\}$. 
%\end{remark}
%%%%%%%%%   ここまで消さない!!!!!!!!!!!!!!
\begin{corollary}\label{toric sm}
The invariant Hilbert scheme $\mathcal H$ is smooth, and it coincides with the main component $\mathcal H^{main}$. 
\end{corollary}
\begin{proof}
Taking Theorems \ref{tangent} and \ref{Borel} into account, it suffices to show that 
\[
\dim \Hom^{\G}_S(J_0, A/J_0)=\dim \mathcal H^{main}=3.
\]
Let $\phi \in \Hom^{\G}_S(
J_0,A/J_0)$. 
Since $\phi$ is $\G$-equivariant, we have 
$\phi(X_0^{q-p})=\alpha_1 X_1X_3$, 
$\phi(X_2)=\alpha_2 X_1$, $\phi(X_4)=\alpha_3 X_3$, 
and $\phi(X_1^{aq}X_3^{ap})=\alpha_4$ 
for some $\alpha_1,\; \alpha_2,\; \alpha_3,\; \alpha_4 \in \C$. 
Also, since $\phi$ is a homomorphism of $S$-modules, we have 
\[
0=\phi(X_0^{q-p}-X_1X_4+X_2X_3) = \phi(X_0^{q-p})-X_1\phi(X_4)+\phi(X_2)X_3
= (\alpha_1+\alpha_2-\alpha_3)X_1X_3.
\]
Notice that $X_1X_3 \notin J_0$, since otherwise we have $\dim (A/J_0)_{(q-p,0)}=0$, which contradicts to $h(q-p,0)=1$. 
%and therefore $\alpha_1+\alpha_2-\alpha_3=0$. 
Therefore we have $\dim \Hom^{\G}_S(J_0, A/J_0) \leq 3$, and hence the equality. 
\end{proof}
%%%%%%
\section{The main component}\label{s-main}
In this last section, we decide the main component $\mathcal H^{main}$.  
\begin{theorem}\label{maintheorem}
The invariant Hilbert scheme $\mathcal H$ is isomorphic to the blow-up $Bl_O(\E)$ of $\E$ at the origin. 
\end{theorem}
As we have seen in  \S \ref{s-hilb}, we can construct an $SL(2)$-equivariant morphism 
\[
\xymatrix{
\eta_{-p, -1} : \mathcal H \ar[r] & \Gr(1, F_{-p, -1}^{\vee}) \cong \mathbb P^1,
}
\]
where the isomorphism $\Gr(1, F_{-p, -1}^{\vee}) \cong \mathbb P^1$ is given by 
$\langle t_0 {X_1}^{\vee} + t_1 {X_2}^{\vee}\rangle \mapsto [t_0:t_1]$.  
Analogously, we have 
\[
\xymatrix{
\eta_{q,1} : \mathcal H \ar[r] & \Gr(1, F_{q,1}^{\vee}) \cong \mathbb P^1.
}
\]
%\begin{lemma}\label{P1}
%There is an $SL(2)$-equivariant morphism 
%\[
%\eta : \H \to \P.
%\]
%\end{lemma}
%\begin{proof}
%By Lemma \ref{generator} (i), the $SL(2)$-module $F_{(-p, m-1)}=<X_1, X_2>$ 
%generates 
%$S_{(-p, m-1)}$ 
%over 
%the invariant ring $S^{\G}$. 
%Therefore, as we have seen in \S 1, we can construct an $SL(2)$-equivariant morphism 
%\[
%\eta_{(-p, m-1)} : \H \to \Gr(F_{(-p, m-1)}^{\vee}, 1).
%\]
%Let
%\[
%\eta_ {(-p, m-1)}' : \H \to \mathbb P^1
%\]
%be the composition of 
%$\eta_ {(-p, m-1)}$ and the following $SL(2)$-equivariant isomorphism:
%\[
%\xymatrix@C=36pt@R=-5pt{
%\Gr(F_{(-p, m-1)}^{\vee}, 1) \ar[r]^{\;\; \qquad \sim} & \mathbb P^1 \\
%\rotatebox{90}{$\in$} & \rotatebox{90}{$\in$} \\
%<t_0 {X_1}^{\vee} + t_1 {X_2}^{\vee}> \ar@{|->}[r] & [t_0: t_1]
%}
%\] 
%Analogously, we can construct the following $SL(2)$-equivariant morphism:
%\[
%\xymatrix{
%\eta_{(q,1)}' : \H \ar[r]^{\; \; \qquad \eta_{(q,1)}} & \Gr(F_{(q,1)}^{\vee}, 1) \ar[r]^{\;\; \quad \sim} & \mathbb P^1.
%}
%\]
%Then, $\eta:=\eta_{(-p, m-1)}' \times \eta_{(q,1)}'$ 
%is the desired morphism. 
%\end{proof}
%Here we concretely describe the action of $SL(2)$ on $\P$ appeared in Lemma \ref{P1}. 
%Let $g=
%\begin{pmatrix}
%\lambda_{11} & \lambda_{12} \\
%\lambda_{21} & \lambda_{22}
%\end{pmatrix}
%\in SL(2)
%$, and $([t_0 : t_1], [s_0 : s_1])$ be a coordinate of $\P$. 
%Then, $g \cdot ([t_0 : t_1], [s_0 :  s_1])$ is given by 
%\[
%([\lambda_{22}t_0-\lambda_{12}t_1 : -\lambda_{21}t_0+\lambda_{11}t_1], 
%[\lambda_{22}s_0-\lambda_{12}s_1 : -\lambda_{21}s_0+\lambda_{11}s_1]).
%\]
Let $\eta=\eta_{-p,-1} \times \eta_{q,1}$, and define 
\[
\xymatrix{
\psi=\gamma \times \eta : \mathcal H \ar[r] & E_{l,m} \times \mathbb P^1 \times 
\mathbb P^1.
}
\]
We see that $\mathbb P^1 \times \mathbb P^1$ contains exactly two orbits under the induced $SL(2)$-action. 
Indeed, let 
$y_1:=([1:0], [0:1])$, and $y_2:=([1:0], [1:0])$. 
Then, the $SL(2)$-orbit decomposition is given as 
$\P=\mathcal O_1 \sqcup \mathcal O_2$, 
where 
\[
\mathcal O_1:=SL(2) \cdot y_1 \cong SL(2)/T, \quad \mathcal O_2:=SL(2) \cdot y_2 \cong SL(2)/B.
\] 
\begin{remark}
The construction of $\psi$ is valid without the toric hypothesis on $\E$. 
\end{remark}
%We see that $\mathcal O_1 \cong SL(2)/T$, $\mathcal O_2 \cong SL(2)/B \cong \mathbb P^1$, and that $\mathcal O_2$ is diagonally embedded into $\mathbb P^1 \times \mathbb P^1$. 
Now, following the idea of \cite[\S 4.3]{Bec}, we get:
\begin{lemma}\label{bij}
The morphism $\psi$ is bijective onto its image. 
\end{lemma}
\begin{proof}
%Let $\mathcal H_i:=\eta^{-1}(\mathcal O_i)$ for $i=1, 2$.
%Then
Consider the following $SL(2)$-equivariant commutative diagram:
\[
\xymatrix{
\mathcal H \ar[rr]^{\psi \qquad } \ar[rrd]_{\eta} &  &  E_{l,m} \times \mathbb P^1 \times \mathbb P^1 \ar[d]^{\pr} \\
 &  & \mathbb P^1 \times \mathbb P^1
} 
\]
For $i \in \{1,2\}$, let $\mathcal H_i=\eta^{-1}(\mathcal O_i)$, $L_i=\eta^{-1}(y_i)$, and $N_i=\pr^{-1}(y_i) \cap \psi(\mathcal H)$. 
%, it suffices to show that $\psi |_{\mathcal H_i} : \mathcal H_i \to \psi(\mathcal H_i)$ is bijective for each $i=1, 2$. 
%Let $L_i=\eta|_{\mathcal H_i}^{-1}(y_i)$, and $N_i=pr_2^{-1}(y_i)$. 
Then, we see that  both 
$\mathcal H_i$ and $\psi(\mathcal H_i)$ have a  description as a principal fiber bundle: we have 
\[
\mathcal H_1 \cong SL(2) \times_{T} L_1,\quad 
\psi(\mathcal H_1) \cong SL(2) \times_{T} N_1, 
\]
and 
\[
\mathcal H_2 \cong SL(2) \times_{B} L_2,\quad  
\psi(\mathcal H_2) \cong SL(2) \times_{B} N_2.
\]
Therefore we are left to show that $L_i \to N_i$ is bijective.  
First, suppose that $i=1$, and let $[Z] \in L_1$. 
Then by the construction of $\eta$ we have $X_2,\; X_3 \in I_Z$, and hence $(X_0^{q-p}-X_1X_4,\; X_2,\; X_3) \subset I_Z$. 
%\[
%L_1=\{ [Z] \in \mathcal H_1 \; :\; X_2, X_3 \in I_Z\}.
%\] %
%Let . 
In view of Remark \ref{invariant ring},  
the condition $\dim \C[Z]^{\G}=h(0,0)=1$ implies that 
$s-X_0^{mp}X_1^m \in  I_Z$
for some $s \in \C$. 
Therefore we have $I_s \subset I_Z$, and hence $I_s=I_Z$ concerning Theorem \ref{idealU}. 
Thus we get 
\[
L_1=\{[Z] \in \mathcal H_1\; :\; I_Z=I_{s},\; \exists s \in \C\}.
\] 
For any $s,\; s' \in \C$, we see that $\gamma([I_s])=\gamma([I_{s'}])$ holds if and only if $s=s'$, since  $s-X_0^{mp}X_1^m,\; s'-X_0^{mp}X_1^m \in S^{\G} \cong \C[\E]$. 
Therefore $L_1 \to N_1$ is bijective. 
Likewise, we have 
$L_2=\{
[Z] \in \mathcal H_2\; :\; I_Z=J_s,\; \exists s \in \C\}$,  
and hence $L_2 \to N_2$ is also bijective. 
\end{proof}
\begin{proof}[Proof of Theorem \ref{maintheorem}]
Set $\tilde{\sigma}=\id_{E_{l,m}} \times \sigma$. 
Then we have the following equivariant commutative diagram: 
\[
\xymatrix{
\mathcal H \ar[r]^{\psi \quad \; \; \; \; \; \;} \ar[drr]_{\gamma} & E_{l,m} \times \mathbb P^1 
\times \mathbb P^1 \ar@{^{(}-_>}[r]^{\tilde{\sigma}\quad \; \; } & E_{l,m} \times \mathbb P^{(aq+1)(ap+1)-1}  \ar[d]^{\pr}\\
& & E_{l,m}
}
\]
Also, let $\varphi : Bl_{O}(E_{l,m}) \to E_{l,m}$ 
be the blow-up morphism of $E_{l,m}$ at the origin. 
%Then we have $Bl_{O}(E_{l,m})=\overline{\varphi^{-1}(\mathfrak U)}$. 
By 
the construction of the morphisms $\psi, \sigma$, and $\varphi$, we see that $\varphi^{-1}(\mathfrak U) \cong \tilde{\sigma}(\psi(\gamma^{-1}(\mathfrak U)))$, since $E_{l,m}$ is isomorphic to the affine cone over the embedding $\sigma$ (see Remark \ref{toric case}). 
%Note that we have $\mathcal H=\overline{\gamma^{-1}(\mathfrak U)}$ by Corollary \ref{toric sm}. 
Therefore, by the properness of $\tilde{\sigma} \circ \psi$, we have 
$\tilde{\sigma}(\psi(\mathcal H))=\overline{\tilde{\sigma}(\psi(\gamma^{-1}(\mathfrak U)))} \cong 
Bl_{O}(E_{l,m})$.
Since $Bl_{O}(E_{l,m})$ is normal, it follows from Lemma \ref{bij} and the Zariski's Main Theorem that $\psi$ is a closed immersion. 
\end{proof}
%\begin{remark}
%If $E_{l,m}$ is non-toric, we can see that we need a further blow-up 
%by an easy observation. We will treat the non-toric case in \cite{Ku}. 
%\end{remark}
%\paragraph{Acknowledgement}  
%The author is grateful to professor Yasunari Nagai, 
%her adviser, 
%for his valuable comments and suggestions. 

%The author is grateful to Professor Yasunari Nagai, her advisor, for 
%his valuable suggestions and continued encouragement. 
%She would also like to thank Professor Hajime Kaji for his helpful 
%advice and comments. 

\paragraph{Acknowledgement}  
The author would like to express her gratitude to Professor Yasunari Nagai, her supervisor, for his valuable discussion and continued unwavering encouragement. 
She is also grateful to Professor Hajime Kaji for his beneficial  advice and kind support. 
Lastly, she would like to thank Professors Daizo Ishikawa, Ryo Ohkawa, and Taku Suzuki for their useful comments and helpful suggestions.

\end{document}